\documentclass[11pt]{article}

 \usepackage{amsmath,amsthm,amsfonts,amssymb}
 \usepackage{graphicx}
 \usepackage{color}

\numberwithin{equation}{section}

\newcommand{\ind}{{\bf 1}}
\newcommand{\proba}{\mathbb P}
\newcommand{\esp}{{\mathbb E}}
\newcommand{\supp}{{\rm{supp}}}

\newcommand{\defe}{\mathrel{\mathop:}=}

\newcommand{\inv}{^{-1}}

\newcommand{\calB}{{\cal B}}

\newcommand{\filF}{{\cal F}}
\newcommand{\calG}{{\cal G}}

\newcommand{\calI}{{\cal I}}

\newcommand{\calP}{{\cal P}}

\def\indt#1{\{#1_t\}_{t\in T}}

\def\lap{{L^\alpha_+}}

\newcommand{\eqnh}{\begin{eqnarray*}}
\newcommand{\eqne}{\end{eqnarray*}}
\newcommand{\eqnhn}{\begin{eqnarray}}
\newcommand{\eqnen}{\end{eqnarray}}
\newcommand{\equh}{\begin{equation}}
\newcommand{\eque}{\end{equation}}

\def\summ#1#2#3{\sum_{#1 = #2}^{#3}}

\newcommand{\eqd}{\stackrel{\rm d}{=}}

\newcommand{\widebar}{\overline}

\def\Eint#1{\, \int^{\!\!\!\!\!\!\!\!{e}}_{#1}}

\def\topp#1{^{(#1)}}

\def\ddfrac#1#2#3{\frac{\d(#1\circ #2)}{\d #3}}

\def\lap{L^\alpha_+}

\def\nn#1{{\left\|#1\right\|}}

\def\babs#1{\Big|#1\Big|}
\def\ccbb#1{\left\{#1\right\}}
\def\bccbb#1{\Big\{#1\Big\}}
\def\sccbb#1{\{#1\}}
\def\bpp#1{\Big(#1\Big)}
\def\spp#1{(#1)}

\def\vv#1{{\bf #1}}

\def\d{{\rm d}}

\def\bbR{{\mathbb R}}

\def\P{{\mathbb P}}

\def\mfa{\mbox{ for all }}

\def\adaptF#1{\{#1_t,\filF_t:0\leq t<\infty\}}

\def\wt#1{\widetilde{#1}}
\def\wb#1{\widebar{#1}}

\def\ifhead#1#2{\left\{\begin{array}{#1@{\quad\mbox{ if }\quad}#2}}
\def\ifend{\end{array}\right.}

\def\ddfrac#1#2#3{\frac{\d#1\circ#2}{\d #3}}
\def\comment#1{}
\renewcommand\cal{\mathcal }
\renewcommand{\Bbb}{\mathbb}

\newcommand{\bbZ}{{\Bbb Z}}

\renewcommand\P{\Bbb P}

\renewcommand{\supp}{\mbox{{\rm supp}}}

\def\Eint{{\int^{\!\!\!\!\!\!\!{\rm e}}}}

 \newtheorem{theorem}{\bf Theorem}[section]
 \newtheorem{lemma}{\bf Lemma}[section]
 
 \newtheorem{corollary}{\bf Corollary}[section]

 \newtheorem{proposition}{\bf Proposition}[section]

 \newtheorem{definition}{\bf Definition}[section]

\theoremstyle{definition}
 \newtheorem{remark}{\bf Remark}[section]
 \newtheorem{Example}{\bf Example}[section]

\title{Decomposability for Stable Processes}
\author{Yizao Wang\thanks{
 Department of Statistics, The University of Michigan,
 439 W.\ Hall, 1085 S.\ University, Ann Arbor, MI 48109--1107. U.S.A.
  {\em E-mails:}
  \texttt{yizwang, sstoev@umich.edu}.
  } \thanks{Corresponding author. Fax: +1 7347634676},\ \ Stilian A.~Stoev$^*$\ \ and\ \  Parthanil Roy\thanks{Statistics and Mathematics Unit, Indian Statistical Institute, Kolkata 700108, India. {\em E-mail:} {\tt parthanil.roy@gmail.com}.
  }}
\begin{document}\sloppy
\maketitle

\begin{abstract}

\noindent We characterize all possible independent symmetric $\alpha$-stable (S$\alpha$S) 
components of an S$\alpha$S process, $0<\alpha<2$. 
In particular, we focus on stationary S$\alpha$S processes and their independent stationary S$\alpha$S components. We also develop a parallel characterization theory for max-stable processes. \medskip


\end{abstract}
\section{Introduction}

Recall that a random variable $Z$ has a {\em symmetric $\alpha$-stable} (S$\alpha$S) distribution with $0<\alpha\le 2$, 
if $\esp\exp(itZ) = \exp(-\sigma^\alpha|t|^\alpha)$ for all $t\in\mathbb R$ with some constant $\sigma>0$. A process 
$X=\indt X$ is said to be S$\alpha$S if all its finite linear combinations follow S$\alpha$S distributions.

In this paper, we investigate the general decomposability problem for S$\alpha$S processes with $0<\alpha<2$. Namely, let $X = \indt X$ be an S$\alpha$S process indexed by an arbitrary set $T$. Suppose that
\begin{equation}\label{eq:decomposition}
 \{X_t\}_{t\in T} \eqd \bccbb{ X_t^{(1)} + \cdots + X_t^{(n)}}_{t\in T},
\end{equation}
where `$\eqd$' means equality in finite-dimensional distributions, and $X\topp k = \indt{X\topp k}, k=1,\dots,n$
are {\em independent} S$\alpha$S processes. We will write $X\eqd X\topp1+\cdots+X\topp n$ 
in short, and each $X^{(k)}$ will be referred to as a {\em component} of $X$. The stability property readily 
implies that~\eqref{eq:decomposition} holds with $X\topp k\eqd n^{-1/\alpha} X \equiv  \{n^{-1/\alpha}X_t\}_{t\in T}$. 
The components equal in finite-dimensional distributions to a constant multiple of $X$ will be referred to as 
{\em trivial}. We are interested in the general structure of all possible {\it non-trivial} S$\alpha$S components of $X$.

Many important decompositions \eqref{eq:decomposition} of S$\alpha$S processes are already available 
in the literature: see for example Cambanis et al.~\cite{cambanis92characterization}, 
Rosi\'nski~\cite{rosinski95structure},   Rosi\'nski and Samorodnitsky~\cite{rosinski:samorodnitsky:1996}, Surgailis et al.~\cite{surgailis98mixing},  
Pipiras and Taqqu~\cite{pipiras02structure, pipiras:taqqu:2004cy}, and  Samorodnitsky~\cite{samorodnitsky05null}, to name a few. 
These results were motivated by studies of various probabilistic and structural aspects of the underlying 
S$\alpha$S processes such as ergodicity, mixing, stationarity, self-similarity, etc. 
Notably, Rosi\'nski~\cite{rosinski95structure} established a fundamental connection between 
stationary S$\alpha$S processes and non-singular flows. He developed important tools based on minimal 
representations of S$\alpha$S processes and inspired multiple decomposition results motivated by connections 
to ergodic theory.

In this paper, we adopt a different perspective. 
Our main goal is to characterize of
{\it all} possible S$\alpha$S decompositions \eqref{eq:decomposition}. 
Our results show how the dependence structure of an S$\alpha$S process determines the structure of its components. 

Consider S$\alpha$S processes $\indt X$ indexed by a complete separable metric space $T$ with
an integral representation
\equh\label{rep:integral}
 \{X_t\}_{t\in T} \eqd {\Big\{} \int_{S} f_t(s) M_\alpha(\d s) {\Big\}}_{t\in T},
\eque
where real-valued functions $\{f_t\}_{t\in T} \subset L^\alpha(S,\calB_S,\mu)$ are referred to as the {\it spectral functions} of $\indt X$.  By default, $M_\alpha$ is a real-valued S$\alpha$S random measure on the
standard Lebesgue space $(S,\calB_S,\mu)$, with a $\sigma$-finite control measure $\mu$.  
The spectral functions determine the finite-dimensional distributions of the process: for all $n\in\mathbb N, t_j\in T, a_j\in\mathbb R$,
\equh\label{eq:fdd}
\esp\exp\bpp{-i\summ j1na_jX_{t_j}} = \exp\bpp{-\int_S\babs{\summ j1na_jf_{t_j}}^\alpha\d\mu}\,.
\eque
Every separable in
probability S$\alpha$S process $X$ can be shown to have such a representation; see, for example, the excellent book by Samorodnitsky and Taqqu~\cite{samorodnitsky94stable} for detailed discussions on S$\alpha$S distributions and processes. Without loss of generality, we always assume that the spectral functions $\indt f\subset L^\alpha(S,\calB_S,\mu)$ have {\em full support}, i.e., $S = {\rm supp}\{ f_t,\ t\in T\}$.

We first state the main result of this paper. To this end, we recall that the {\it ratio $\sigma$-algebra} of a spectral representation $F= \indt f$ (of $\{X_t\}$) is defined as
\begin{equation}\label{e:rho}
 \rho(F) \equiv\rho\{ f_t,\ t\in T\} := \sigma\{ f_{t_1}/f_{t_2},\ t_1,t_2\in T\}.
\end{equation}
The following result characterizes the structure of all S$\alpha$S decompositions. 
\begin{theorem}\label{thm:1}
Suppose $\indt X$ is an S$\alpha$S process ($0<\alpha<2$) with spectral representation
\[
\indt X \eqd \bccbb{\int_Sf_t(s)M_\alpha(\d s)}_{t\in T}\,,
\]
with
$\indt f\subset L^\alpha(S,\calB_S,\mu)$. Let $\{X_t^{(k)}\}_{t\in T},\ k=1,\cdots,n$ be independent S$\alpha$S processes.  
\begin{itemize}
\item [(i)] The
decomposition
\equh\label{eq:decomposition1}
\indt X \eqd \ccbb{X\topp 1_t+\cdots + X\topp n_t}_{t\in T}
\eque
holds, if and only if there exist measurable functions
$r_k:S\to[-1,1]$, $k=1,\cdots,n$, such that
\equh\label{eq:Xk}
 \indt{X\topp k}\eqd\bccbb{\int_Sr_k(s)f_t(s)M_\alpha(\d s)}_{t\in T},\ \ k=1,\cdots,n.
\eque
In this case, necessarily $\summ k1n|r_k(s)|^\alpha = 1$, $\mu$-almost everywhere on $S$.
\item[(ii)] If \eqref{eq:decomposition1} holds, then the $r_k$'s in \eqref{eq:Xk} can be chosen
to be non-negative and $\rho(F)$-measurable. Such $r_k$'s are unique modulo $\mu$.
\end{itemize}
\end{theorem}

As an application, we study the structure of the {\em stationary} S$\alpha$S components of a stationary S$\alpha$S
 process. 
We obtain a characterization for all possible stationary components of stationary S$\alpha$S processes
in Theorem~\ref{thm:stationary} below.
As a simple example, consider the moving average process $\{X_t\}_{t\in\mathbb R^d}$ with spectral representation
\[
\{X_t\}_{t\in\mathbb R^d}\eqd \bccbb{\int_{\mathbb R^d}f(t+s)M_\alpha(\d s)}_{t\in \mathbb R^d},
\]
where $d\in\mathbb N$, $M_\alpha$ is an S$\alpha$S random measure on $\mathbb R^d$ with the Lebesgue control measure $\lambda$, and $f\in L^\alpha(\mathbb R^d,\calB_{\mathbb R^d},\lambda)$ (see, e.g.,~\cite{samorodnitsky94stable}). We show that such a process has only trivial stationary S$\alpha$S components, i.e.\
all its stationary components are rescaled versions of the original process (Corollary~\ref{coro:MMA1}). Such stationary 
S$\alpha$S processes will be called {\em indecomposable}. 
More examples are provided in Sections~\ref{sec:characterization} and~\ref{sec:stationary}.

We also develop parallel decomposability theory for max-stable
processes. Recently, Kabluchko~\cite{kabluchko09spectral} and Wang and Stoev~\cite{wang10association, wang10structure}
have established intrinsic connections between sum- and
max-stable processes. In particular, the tools in  \cite{wang10association} readily imply that the developed decomposition
theory for S$\alpha$S processes applies {\em mutatis mutandis} to max-stable processes. 

{ The rest of the paper is structured as follows.} In Section \ref{sec:characterization}, we provide some consequences of Theorem~\ref{thm:1} for general S$\alpha$S processes. The stationary case is discussed in Section \ref{sec:stationary}. Parallel results on max-stable processes are presented in Section \ref{sec:maxstable}. The proof of Theorem~\ref{thm:1} is given in Section~\ref{sec:proofs}.
%
%
%
%
%
\section{S$\alpha$S Components} \label{sec:characterization}
In this section, we provide a few examples to illustrate the consequences of our main result Theorem~\ref{thm:1}. The first one is about S$\alpha$S processes with independent increments. Recall that we always assume $0<\alpha<2$.

\begin{corollary}\label{c:ind-incr} Let $X=\{X_t\}_{t\in\mathbb R_+}$ be an arbitrary S$\alpha$S process with
independent increments and $X_0 = 0$.  Then all S$\alpha$S components of $X$ also have independent increments.
\end{corollary}
\begin{proof} Write $m(t) = \nn{X_t}_\alpha^\alpha$, where $\|X_t\|_\alpha$ denotes the scale coefficient of
the S$\alpha$S random variable $X_t$. By the independence of the increments of $X$, it follows that $m$ is a 
non-decreasing function with $m(0) = 0$. First, we consider the simple case when $m(t)$ is right-continuous. 
Consider the Borel measure $\mu$ on $[0,\infty)$ determined by $\mu([0,t]):= m(t)$. The independence of the
increments of $X$ readily implies that $X$ has the representation:
\equh\label{eq:ind}
 \displaystyle{\{X_t\}_{t\in\mathbb R_+} \eqd {\Big\{} \int_0^\infty\ind_{[0,t]}(s) M_\alpha(\d s) {\Big\}}_{t\in\mathbb R_+}},
\eque
where $M_\alpha$ is an S$\alpha$S random measure with control measure $\mu$.

Now, for any S$\alpha$S component $Y (\equiv X^{(k)})$ of $X$, we have that \eqref{eq:Xk} holds
with $f_t(s) = \ind_{[0,t]}(s)$ and some function $r(s) (\equiv r_k(s))$.  This implies that the increments of $Y$ are also
independent since, for example, for any $0\le t_1<t_2 $, the spectral functions $r(s) f_{t_1}(s) =
r(s) \ind_{[0,t_1]}(s)$ and $r(s) f_{t_2}(s) - r(s) f_{t_1}(s) = r(s) \ind_{(t_1,t_2]}(s)$ have disjoint supports.

It remains to prove the general case. The difficulty is that $m(t)$ may have (at most countably many) discontinuities, and a representation as~\eqref{eq:ind} is not always possible. Nevertheless, introduce the right-continuous functions $t\mapsto m_i(t), i=0,1$,  
\[
m_0(t) \defe m(t+) - \sum_{\tau\leq t}(m(\tau)-m(\tau-))\ \ \mbox{ and }\ \
m_1(t) \defe \sum_{\tau\leq t}(m(\tau)-m(\tau-))
\]
and let $\wt M_\alpha$ be an S$\alpha$S random measure on $\mathbb R_+\times\{0,1\}$ with control 
measure $\mu([0,t]\times\{i\}) := m_i(t),\ i=0,1,\ t\in\bbR_+$. In this way, as in \eqref{eq:ind} one can show that
\[
\indt X \eqd \bccbb{\int_{\mathbb R_+\times\{0,1\}}\ind_{[0,t)\times\{0\}}(s,v) + \ind_{[0,t]\times\{1\}}(s,v) \wt M_\alpha(\d s,\d v)}_{t\in T}\,.
\]
The rest of the proof remains similar and is omitted.
\end{proof}
\begin{remark}
Theorem~\ref{thm:1} and Corollary~\ref{c:ind-incr} do not apply to the Gaussian case ($\alpha = 2$). For the sake of simplicity, take $T = \{1,2\}$ and $n=2$ (2 S$\alpha$S components) in~\eqref{eq:decomposition}. In this case, all the (in)dependence information of the mean-zero Gaussian process $\indt X$ is characterized by the covariance matrix $\Sigma$ of the Gaussian vector $(X_1\topp1, X_1\topp2, X_2\topp1,X_2\topp2)$. A counterexample can be easily constructed by choosing appropriately $\Sigma$.  This reflects the drastic difference of the geometries of $L^\alpha$ spaces for $\alpha<2$ and $\alpha = 2$. 
\end{remark}

The next natural question to ask is whether two S$\alpha$S processes have {\it common components}. Namely, the S$\alpha$S process $Z$ is a common component of the S$\alpha$S processes $X$ and $Y$, if $X\eqd Z+X\topp1$ and $Y\eqd Z+Y\topp1$, where ${X\topp1}$ and ${Y\topp1}$ are both S$\alpha$S processes independent of $Z$. 

To study the common components, the {\it co-spectral} point of view introduced in Wang and Stoev~\cite{wang10structure} is helpful.
Consider a {\it measurable} S$\alpha$S process $\indt X$ with spectral representation~\eqref{rep:integral}, where the index set $T$ is equipped with a measure $\lambda$ defined on the $\sigma$-algebra $\calB_T$. Without loss of generality, we take $f(\cdot,\cdot):(S\times T, \calB_S\times\calB_T)\to (\mathbb R,\calB_\mathbb R)$ to be jointly measurable (see Theorems 9.4.2 and 11.1.1 in~\cite{samorodnitsky94stable}). The {\it co-spectral functions}, $f_\cdot(s)\equiv f(s,\cdot)$, are elements of $L^0(T)\equiv L^0(T,\calB_T,\lambda)$, the space of $\calB_T$-measurable functions modulo $\lambda$-null sets. The co-spectral functions are indexed by $s\in S$, in contrast to the spectral functions $f_t(\cdot)$ indexed by $t\in T$.
Recall also that a set $\calP\subset L^0(T)$ is a {\it cone}, if $c\calP = \calP$ for all $c\in\mathbb R\setminus\{0\}$ and $\{\vv 0\}\in\calP$. We write $\{f_\cdot(s)\}_{s\in S}\subset\calP$ modulo $\mu$, if for $\mu$-almost all $s\in S$, $f_\cdot(s)\in\calP$.

\begin{proposition}\label{prop:common}
Let $X^{(i)}=\{X_t^{(i)}\}_{t\in T}$ be S$\alpha$S processes with measurable representations
$\{f_t^{(i)}\}_{t\in T}\subset L^\alpha(S_i,\calB_{S_i},\mu_i),\ i=1,2$.  If there exist two cones $\calP_i\subset L^0(T),i=1,2$, such that $\{f\topp i_\cdot(s)\}_{s\in S_i}\subset\calP_i$ modulo $\mu_i$, for $i = 1,2$, and $\calP_1\cap\calP_2 = \{\vv 0\}$, then the two processes have no common component.
\end{proposition}
\begin{proof}
Suppose $Z$ is a component of $X\topp 1$. Then, by Theorem~\ref{thm:1}, $Z$ has a spectral representation $\indt{r\topp1f\topp1}$, for some $\calB_{S_1}$-measurable function $r\topp1$.
By the definition of cones, the co-spectral functions of $Z$ are included in $\calP_1$, i.e., $\{r\topp1(s)f_\cdot\topp1(s)\}_{s\in S_1}\subset\calP_1$ modulo $\mu_1$. If $Z$ is also a component of $X\topp 2$, then by the same argument, $\{r\topp2(s)f_\cdot\topp2(s)\}_{s\in S_2}\subset\calP_2$ modulo $\mu_2$, for some $\calB_{S_2}$-measurable function $r\topp2(s)$. Since $\calP_1\cap\calP_2 = \{\vv 0\}$, it then follows that $\mu_i(\supp(r\topp i)) = 0, i=1,2$, or equivalently $Z = 0$, the degenerate case.
\end{proof}

We conclude this section with an application to S$\alpha$S moving averages.

\begin{corollary}\label{coro:MAequal}
Let $X\topp1$ and $X\topp2$ be two S$\alpha$S moving averages
\[
\{X\topp i_t\}_{t\in\bbR^d}\eqd \bccbb{\int_{\bbR^d}f\topp i(t+s)M_\alpha\topp i(\d s)}_{t\in \bbR^d}
\]
 with kernel functions  $f\topp i\in L^\alpha(\mathbb R^d,\calB_{\mathbb R^d},\lambda), i=1,2$. Then, either
\equh\label{eq:MAequal}
X\topp1\eqd cX\topp2\mbox{ for some } c>0\,,
\eque
or $X\topp1$ and $X \topp2$ have no common component. Moreover,~\eqref{eq:MAequal} holds, if and only if for some $\tau\in\mathbb R^d$ and $\epsilon\in\{\pm 1\}$,
\equh\label{eq:MAequal1}
f\topp1(s) = \epsilon cf\topp2(s+\tau)\,, \mu\mbox{-almost all } s\in S.
\eque
\end{corollary}
\begin{proof}
Clearly~\eqref{eq:MAequal1} implies~\eqref{eq:MAequal}. Conversely, if~\eqref{eq:MAequal} holds, then~\eqref{eq:MAequal1} follows as in the proof of Corollary 4.2 in~\cite{wang10structure}, with slight modification (the proof therein was for {\it positive} cones). When~\eqref{eq:MAequal} (or equivalently~\eqref{eq:MAequal1}) does not hold, consider the smallest cones containing $\{f\topp i(s+\cdot)\}_{s\in\mathbb R}, i =1,2$ respectively. Since these two cones have trivial intersection $\{\vv 0\}$, Proposition~\ref{prop:common} implies that $X\topp1$ and $X\topp 2$ have no common component.
\end{proof}


\section{Stationary S$\alpha$S Components and Flows}\label{sec:stationary}

Let $X =\{X_t\}_{t\in T}$ be a stationary S$\alpha$S process with representation \eqref{rep:integral},
where now $T =\bbR^d$ or $T=\bbZ^d$, $d\in\mathbb N$.
The seminal work of Ros\'nski~\cite{rosinski95structure} established an important connection between stationary S$\alpha$S processes and {\it flows}.
A family of functions $\indt\phi$ is said to be a flow on $(S,\calB_S,\mu)$, if for all $t_1,t_2\in T$, $\phi_{t_1+t_2}(s) = \phi_{t_1}(\phi_{t_2}(s))$ for all $s\in S$, and $\phi_0(s) = s$ for all $s\in S$. We say that a flow is {\it non-singular}, if $\mu(\phi_t(A)) = 0$ is equivalent to
$\mu(A) = 0$, for all $A\in\calB_S, t\in T$. Given a flow $\indt\phi$, $\indt c$ is said to be a {\it cocycle} if $c_{t+\tau}(s) = c_t(s)c_\tau\circ\phi_t(s)$ $\mu$-almost surely for all $t,\tau\in T$ and $c_t\in\{\pm1\}$ for all $t\in T$.

To understand the relation between the structure of stationary S$\alpha$S processes and flows, it is necessary to work with {\em minimal}
representations of S$\alpha$S processes, introduced by Hardin~\cite{hardin81isometries,hardin82spectral}. 
The minimality assumption is crucial in many results on the structure of
S$\alpha$S processes, although it is in general difficult to check (see e.g.~Rosi\'nski~\cite{rosinski06minimal} and 
Pipiras~\cite{pipiras07nonminimal}). 

\begin{definition} \label{d:minimal} The spectral functions $F \equiv \indt f$ (and the corresponding spectral representation~\eqref{rep:integral}) are
said to be minimal, if the ratio $\sigma$-algebra $\rho(F)$ in \eqref{e:rho}
is equivalent to $\calB_S$, i.e., for all $A\in \calB_S$, there exists $B \in \rho(F)$ such that
$\mu(A\Delta B) = 0,$ where $A\Delta B = (A\setminus B) \cup (B\setminus A)$.
\end{definition}

Rosi\'nski (\cite{rosinski95structure}, Theorem 3.1) proved that if $\indt f$ is minimal, then there exists a modulo
$\mu$ unique non-singular flow $\indt\phi$, and a corresponding cocycle
$\indt c$, such that for all $t\in T$,
\equh\label{eq:flow}
f_t(s) = c_t(s)\bpp{\ddfrac\mu{\phi_t}\mu(s)}^{1/\alpha}f_0\circ\phi_t(s)\,, \mu\mbox{-almost everywhere.}
\eque

Conversely, suppose that \eqref{eq:flow} holds for some non-singular flow $\indt \phi$, a corresponding cocycle $\indt c$, and a function
$f_0\in L^\alpha(S,\mu)$ ($\indt f$ not necessarily minimal).  Then, clearly the S$\alpha$S process $X$ in \eqref{rep:integral} is stationary.  In this case, we
shall say that $X$ is generated by the flow $\indt \phi$.

Consider now an S$\alpha$S decomposition \eqref{eq:decomposition} of $X$, where the independent components $\indt {X\topp k}$'s
are {\em stationary}.  This will be referred to as a {\em stationary S$\alpha$S decomposition}, and the $\indt {X\topp k}$'s
as {\em stationary components} of $X$.  Our goal in this section is to characterize the structure of all possible stationary components.
This characterization involves the invariant $\sigma$-algebra with respect to the flow $\indt\phi$:
\begin{equation}\label{e:F-phi}
\filF_\phi = \{A\in\calB_S:\mu(\phi_\tau(A) \Delta A) =0\,,\mfa \tau\in T\}\,.
\end{equation}
Given a function $g$ and a $\sigma$-algebra $\calG$, we write $g\in\calG$, if $g$ is measurable with respect to $\calG$. 
\begin{theorem}\label{thm:stationary} Let $\indt X$ be a stationary and measurable
S$\alpha$S process with spectral functions $\indt f$ given by 
\[
f_t(s) = \int_S c_t(s)\bpp{\ddfrac\mu{\phi_t}\mu(s)}^{1/\alpha}f_0\circ\phi_t(s)M_\alpha(\d s), t\in T\,.
\]
{\it (i)} Suppose that $\indt X$ has a stationary S$\alpha$S decomposition
\equh\label{eq:decomposition_stat}
\indt X \eqd \ccbb{X\topp 1_t+\cdots + X\topp n_t}_{t\in T}.
\eque

Then, each component $\indt {X\topp k}$ has a representation
\equh\label{eq:Xk_stat}
 \indt{X\topp k}\eqd\bccbb{\int_Sr_k(s)f_t(s)M_\alpha(\d s)}_{t\in T},\ \ k=1,\cdots,n,
\eque
where the 
$r_k$'s can be chosen to be non-negative and $\rho(F)$-measurable.  This choice is unique modulo 
$\mu$ and these $r_k$'s are $\phi$-invariant, i.e.\ $r_k\in\filF_\phi$. 

\noindent{\it (ii)} Conversely, for any $\phi$-invariant $r_k$'s such that $\summ k1n|r_k(s)|^\alpha = 1$, $\mu$-almost everywhere on $S$, decomposition \eqref{eq:decomposition_stat} holds with $X^{(k)}$'s as in
\eqref{eq:Xk_stat}.
\end{theorem}

\begin{proof} By using \eqref{eq:flow}, a change of variables, and the $\phi$-invariance of 
the functions $r_k$'s, one can show that the $X^{(k)}$'s in \eqref{eq:Xk_stat} are stationary. This fact and 
Theorem \ref{thm:1} yield part {\em (ii)}.

We now show {\em (i)}. 
Suppose that $X^{(k)}$ is a stationary (S$\alpha$S) component of $X$.  Theorem \ref{thm:1} implies that there 
exists unique modulo $\mu$ non-negative and $\rho(F)$-measurable function $r_k$ for which \eqref{eq:Xk_stat} holds.  
By the stationarity of $X^{(k)}$, it also follows that for all $\tau \in T$,  $\{ r_k(s) f_{t+\tau}(s)\}_{t\in T}$ is also a 
spectral representation of $X^{(k)}$. By the flow representation \eqref{eq:flow}, it follows that for all $t,\tau\in T$,
\equh\label{eq:flow1}
f_{t+\tau}(s) = c_\tau(s)f_t\circ\phi_\tau(s)\bpp{\ddfrac\mu{\phi_\tau}\mu}^{1/\alpha}(s)\,,\mbox{ $\mu$-almost everywhere,}
\eque
and we obtain that for all $\tau, t_j\in T, a_j\in \mathbb R,\ j=1,\cdots,n$:
\[
\int_S\babs{\summ j1na_jr_k(s)f_{t_j+\tau}(s)}^\alpha\mu(\d s) 
= \int_S {\Big|}\sum_{j=1}^n a_j r_k\circ\phi_{-\tau}(s) f_{t_j}(s) {\Big|} ^\alpha \mu(\d s),\nonumber
\]
which shows that $\{r_k \circ\phi_{-\tau}(s) f_t(s)\}_{t\in T}$ is also a representation for $X^{(k)}$, for all $\tau \in T$.

Observe that from~\eqref{eq:flow1}, for all $t_1,t_2,\tau\in T$ and $\lambda\in\mathbb R$,
\[
\bccbb{\frac{f_{t_1+\tau}}{f_{t_2+\tau}}\leq \lambda} = \phi_\tau\inv\bccbb{\frac{f_{t_1}}{f_{t_2}}\leq \lambda}\mbox{ modulo }\mu.
\]
It then follows that for all $\tau\in T$, the $\sigma$-algebra
$\phi_{-\tau} (\rho(F))\equiv (\phi_\tau)^{-1} (\rho(F))$ is equivalent to $\rho(F)$. This, by the uniqueness of $r_k\in\rho(F)$ (Theorem \ref{thm:1}),
implies that $r_k\circ \phi_\tau = r_k$ modulo $\mu$, for all $\tau$. Then, $r_k\in {\cal F}_\phi$ follows from standard measure-theoretic argument. The proof is complete.
\end{proof}

\begin{remark}
The structure of the {\em stationary} S$\alpha$S components of stationary S$\alpha$S processes (including 
random fields) has attracted much interest since the seminal work of 
Rosi\'nski~\cite{rosinski95structure, rosinski00decomposition}. See, for example, 
Pipiras and Taqqu~\cite{pipiras:taqqu:2004cy}, Samorodnitsky~\cite{samorodnitsky05null}, 
Roy~\cite{roy07ergodic,roy09poisson}, Roy and Samorodnitsky~\cite{roy08stationary}, 
Roy~\cite{roy10ergodic,roy10nonsingular}, and Wang et al.~\cite{wang09ergodic}. 
In view of Theorem~\ref{thm:stationary}, the components considered in these works correspond to indicator functions
$r_k(s) = \ind_{A_k}(s)$ of certain disjoint flow-invariant sets $A_k$'s arising from ergodic theory 
(see e.g.~Krengel~\cite{krengel85ergodic} and Aaronson~\cite{aaronson97introduction}).
\end{remark}

\comment{\begin{remark} Suppose the representation in \eqref{rep:integral} is minimal and consider the
decomposition \eqref{eq:decomposition}, where $X^{(k)}_t := \int_{A_k} f_t(s) M_\alpha(\d s)$ for disjoint 
measurable $A_k$'s with $S = \cup_{k=1}^n A_k$. As in the proof of Corollary \ref{coro:stationaryMinimal}
one can show that the components $X^{(k)} = \{X_t^{(k)}\}_{t\in T},\ k=1,\cdots,n$ are all essentially 
different and non-trivial. 
\end{remark}
}

Theorem \ref{thm:stationary} can be applied to check {\it indecomposability} of stationary  S$\alpha$S processes. 
Recall that a stationary S$\alpha$S process is said to be {\em indecomposable}, if all its stationary S$\alpha$S
components are trivial (i.e.\ constant multiples of the original process). 
\begin{corollary}\label{coro:indecomposable}
Consider $\indt X$ as in Theorem~\ref{thm:stationary}. If $\filF_\phi$ is trivial, then $\indt X$ is indecomposable. The converse is true when, in addition, $\indt f$ is minimal.
\end{corollary}
\begin{proof}
If $\filF_\phi$ is trivial, the result follows from Theorem~\ref{thm:stationary}. 
Conversely, let $\indt f$ be minimal and $X$ indecomposable. Then, one can choose $A\in\filF_\phi$, such that $\mu(A)>0$ and $\mu(S\setminus A)>0$. Then, consider
\[
\indt {X^A}\eqd \bccbb{\int_S\ind_A(s)f_t(s)M_\alpha(\d s)}_{t\in T}.
\]
By Theorem~\ref{thm:stationary}, $X^A$ is a stationary component of $X$.  It suffices to show that $X^A$ is a non-trivial of $X$, which would contradict the indecomposability. 

Suppose that $X^A$ is trivial, then $ cX^A \eqd X$, for some $c>0$.  Thus, by
Theorem \ref{thm:stationary}, $cX^{A}$ has a representation as in \eqref{eq:Xk_stat}, with $r_k:= c\ind_A$.  On the other hand,
since $c X^A \eqd X$, we also have the trivial representation with $r_k:= 1$.  Since $A \in \rho(F)$,
the uniqueness of $r_k$ implies that $1 = c\ind_A$ modulo $\mu$, which contradicts $\mu(A^c)>0$. Therefore, $X^A$ is non-trivial.
\end{proof}
The indecomposable stationary S$\alpha$S
processes can be seen as the elementary building blocks for the construction of
general stationary S$\alpha$S processes. 
We conclude this section with two examples.

\begin{Example}[Mixed moving averages]  \label{example:MMA}
Consider a {\em mixed moving average} in the sense of \cite{surgailis93stable}:
\equh\label{eq:MMA}
\{X_t\}_{t\in \bbR^d}\eqd {\Big\{} \int_{\bbR^d \times V} f(t+s,v) M_\alpha(\d s, \d v) {\Big\}}_{t\in\bbR^d}.
\eque
Here, $M_\alpha$ is an S$\alpha$S random measure on $\bbR^d\times V$ with the control measure $\lambda\times\nu$, where $\lambda$ is the Lebesgue measure on $(\mathbb R^d,\calB_{\mathbb R^d})$ and $\nu$ is a probability measure on $(V,\calB_V)$, and $f(s,v) \in L^\alpha (\bbR^d\times V,\calB_{\bbR^d\times V},\lambda\times\nu)$.
Given a disjoint union $V = \bigcup_{j=1}^nA_j$, where $A_j$'s are measurable subsets of $V$, the mixed moving averages can clearly be decomposed as in~\eqref{eq:decomposition_stat} with
\[
\{X\topp k_t\}_{t\in \mathbb R^d}\eqd \bccbb{\int_{\mathbb R^d\times A_k}f(t+s,v)M_\alpha(\d s,\d v)}_{t\in\mathbb R^d}\,, \mfa k = 1,\dots,n\,.
\]
\end{Example}
Any moving average process
\equh\label{eq:MA}
\{X_t\}_{t\in\bbR^d} \eqd \bccbb{\int_{\bbR^d}f(t+s)M_\alpha(\d s)}_{t\in\bbR^d}
\eque 
trivially has a mixed moving average representation.  
The next result shows when the converse is true.

\begin{corollary}\label{coro:MMA1}
The mixed moving average $X$ in~\eqref{eq:MMA} is indecomposable, if and only if it has a moving
average representation as in~\eqref{eq:MA}.
\end{corollary}

\begin{proof}  By Corollary~\ref{coro:indecomposable}, the moving average process~\eqref{eq:MA} is indecomposable, since in this case $\phi_t(s) = t+s, t,s\in\bbR^d$ and therefore $\filF_\phi$ is trivial. This proves the `if' part.

Suppose now that $X$ in~\eqref{eq:MMA} is indecomposable.
In Section 5 of Pipiras \cite{pipiras07nonminimal}
it was shown that S$\alpha$S processes with mixed moving average representations
and {\em stationary increments} also have minimal representations of the mixed moving average type.
By using similar arguments, one can show that this is also true for the class of {\em stationary}
mixed moving average processes.

Thus, without loss of generality, we assume that the representation in \eqref{eq:MMA} is minimal.
Suppose now that there exists a set $A \in \calB_V$ with
$\nu(A) >0$ and $\nu(A^c)>0$.  Since $\bbR^d\times A$ and $\bbR^d\times A^c$ are flow-invariant, we have
the stationary decomposition
$
\{X_t\}_{t\in \bbR^d} \eqd \{ X_t^A  + X_t^{A^c}\}_{t\in \bbR^d},
$
where
$$
X^B_t := \int_{\bbR\times V}  \ind_{B}(v) f(t+s,v)M_\alpha(\d s,\d v),\ \ B\in \{A, A^c\}.
$$
Note that both components $X^A =\{X_t^{A}\}_{t\in\bbR^d}$ and
$X^{A^c} =\{X_t^{A^c}\}_{t\in\bbR^d}$ are non-zero because the
representation of $X$ has full support.

Now, since $X$ is indecomposable, there exist positive constants $c_1$ and $c_2$, such that
$X\eqd c_1 X^{A} \eqd c_2 X^{A^c}$. The minimality of the representation and
Theorem \ref{thm:stationary} imply that $c_1 \ind_A = c_2 \ind_{A^c}$ modulo $\nu$, which is impossible.
This contradiction shows that the set $V$ cannot be partitioned into two
disjoint sets of positive measure.  That is, $V$ is a singleton and the mixed
moving average is in fact a moving average.
\end{proof}

\begin{Example}[Doubly stationary processes]
Consider a stationary process $\xi=\{\xi_t\}_{t\in T}$  ($T = \bbZ^d$) supported on the probability space $(E,{\cal E},\mu)$ with $\xi_t\in L^\alpha(E,{\cal E},\mu)$. 
Without loss of generality, we may suppose that $\xi_t (u)= \xi_0\circ \phi_t(u)$, where
$\{\phi_t\}_{t\in T}$ is a $\mu$-measure-preserving flow.

Let $M_\alpha$ be an S$\alpha$S random measure on $(E,{\cal E},\mu)$ with control measure $\mu$.  The stationary S$\alpha$S
process $X = \{X_t\}_{t\in T}$
\equh\label{eq:doublyStochastic}
 X_t := \int_{E}\xi_t(u)M_\alpha(\d u), t\in T
\eque
is said to be {\em doubly stationary} (see~Cambanis et al.~\cite{cambanis87ergodic}). 
By Corollary~\ref{coro:indecomposable}, if $\xi$ is ergodic, then $X$ is indecomposable. \end{Example}

A natural and interesting question raised by a referee is: what happens when $X$ is decomposable and hence $\xi$ is non-ergodic? 
Can we have a direct integral decomposition of the process $X$ into indecomposable components? The following remark partly addresses this question.

\begin{remark} The doubly stationary S$\alpha$S processes are a special case of stationary S$\alpha$S processes generated by {\em positively recurrent flows (actions)}.
As shown in Samorodnitsky~\cite{samorodnitsky05null}, Remark 2.6, each such stationary S$\alpha$S process $X = \{X_t\}_{t\in T}$
can be expressed through a measure-preserving flow (action) on a {\em finite} measure space. Namely, 
\begin{equation}\label{e:pos-recurrent}
 \{X_t\}_{t\in T} \eqd {\Big\{} \int_{E} f_t(u) M_\alpha^{(\mu)}(\d u) {\Big\}}_{t\in T}, \ \ \mbox{ with } \ f_t(u):= c_t(u) f_0\circ \phi_t(u),
\end{equation}
where $M_\alpha^{(\mu)}$ is an S$\alpha$S random measure with a {\em finite} control measure $\mu$ on $(E,{\cal E})$, $\phi=\{\phi_t\}_{t\in T}$ is a $\mu$-preserving 
flow (action), and $\{c_t\}_{t\in T}$ is a co-cycle with respect to $\phi$. In the case when the co-cycle is trivial ($c_t\equiv 1$) and $\mu(E) =1$, the process $X$ is {\em doubly stationary}.

For simplicity, suppose that $T = \bbZ^d$ and without loss of generality let $(E,{\cal E},\mu)$ be a standard Lebesgue space with $\mu(E) =1$. 
The ergodic decomposition theorem (see e.g.~Keller~\cite{keller98equilibrium}, Theorem 2.3.3) implies that there exists conditional probability distributions $\{\mu_u\}_{u\in E}$ with respect to $\calI$ such that $\phi$ is measure-preserving and ergodic with respect to the measures $\mu_u$ for $\mu$-almost all $u\in E$.
Let $\nu$ be another $\phi$-invariant measure on $(E,{\cal E})$ dominating the conditional probabilities $\mu_u$ so that the Radon--Nikodym derivatives
$p(x,u) = (\d \mu_u/\d \nu)(x)$ are jointly measurable on $(E\times E, {\cal E}\otimes {\cal E}, \nu\times \mu)$. Consider
\[
g_t(x,u) = f_t(x) p(\phi_t(x),u)^{1/\alpha}.
\]
Recall that $\nu$ and $\mu_u$ are $\phi$-invariant, whence 
$$
 p(\phi_t(x),u) = \frac{\d \mu_{u}}{\d \nu}(\phi_t(x)) = \frac{\d \mu_u}{\d \nu} (x)= p(x,u),\ \ \mbox{ modulo $\nu\times \mu$.}
$$
Thus, $g_t(x,u) =  f_t(x) (\d \mu_u /\d \nu)^{1/\alpha} (x)$, and for all $a_j\in\bbR,\ t_j \in T,\ j=1,\cdots,n$, we have
\begin{eqnarray*}
 \int_{E^2} {\Big|} \sum_{j=1}^n a_j g_{t_j}(x,u){\Big|}^\alpha \nu(\d x)\mu(\d u)\! &=& \!\int_{E^2} {\Big|} \sum_{j=1}^n a_j f_{t_j}(x){\Big|}^\alpha \frac{\d\mu_u}{\d \nu} (x)\nu(\d x) \mu(\d u)\\
& \! = \!&  \int_{E^2} {\Big|} \sum_{j=1}^n a_j f_{t_j}(x){\Big|}^\alpha \d\mu_u (\d x) \mu(\d u) \\&\! =\! &  \int_{E} {\Big|} \sum_{j=1}^n a_j f_{t_j}(x){\Big|}^\alpha \mu(\d x),
\end{eqnarray*}
where the last equality follows from the identity that $\int_{E} h(x) \mu(\d x)  = \int_{E^2} h(x) \mu_u(\d x) \mu(\d u)$, for all $h\in L^1(E,{\cal E},\mu)$.
We have thus shown that $\indt X$ defined by~\eqref{e:pos-recurrent} has another spectral representation
\equh\label{eq:ergoDecomp}
 \indt X\eqd\bccbb{\int_{E \times E} g_t(x,u) M_\alpha^{(\nu\times \mu)}(\d x,\d u)}_{t\in T}\,,
\eque
where $M_\alpha^{(\nu\times \mu)}$ is an S$\alpha$S random measure on $E\times E$ with control measure $\nu\times \mu$.
It also follows that  for $\mu$-almost all $u\in E$, the process defined by
\[
{X^{(u)}_t}\defe{\int_{E}g_t(x,u)M_\alpha^{(\nu)}(\d x)},\  t\in T,
\]
is indecomposable, where $M_\alpha^{(\nu)}$ has control measure $\nu$. Indeed, as above, one can show that 
$$
\{X^{(u)}_t\}_{t\in T} \eqd
{\Big\{}\int_E f_t(u,x) M_\alpha^{(\mu_u)}(\d x){\Big\}}_{t\in T},
$$
where $M_\alpha^{(\mu_u)}$ has control measure $\mu_u$.  The ergodic decomposition theorem implies that the flow (action) $\phi$ is ergodic 
with respect to $\mu_u$, which by Corollary \ref{coro:indecomposable} implies the indecomposability of  $X^{(u)} = \{X^{(u)}_t\}_{t\in T}$. In this way,~\eqref{eq:ergoDecomp} parallels the mixed moving average representation for stationary S$\alpha$S processes generated by {\it dissipative flows} (see e.g.~Rosi\'nski~\cite{rosinski95structure}).
\end{remark}

\begin{remark} The above construction of the decomposition~\eqref{eq:ergoDecomp} assumes the existence of a $\phi$-invariant measure 
$\nu$ dominating all conditional probabilities $\mu_u,\ u\in E$.  If the measure $\mu$, restricted on the invariant $\sigma$-algebra ${\cal F}_\phi$ is discrete, i.e.\ ${\cal F}_\phi$
consists of countably many atoms under $\mu$, then one can take $\nu\equiv \mu$. In this case, the process $X$ is decomposed into a sum (possibly infinite) of its indecomposable
components:
$$
 X_t = \sum_{k} \int_{E_k} f_t(x) M_\alpha^{(\mu)}(\d x),\
$$
where the $E_k$'s are disjoint $\phi$-invariant measurable sets, such that $E = \cup_k E_k$ and $\phi\vert_{E_k}$ is ergodic, for each $k$. In this case, the $E_k$'s are the
atoms of ${\cal F}_\phi$.

In general, when $\mu\vert_{{\cal F}_\phi}$ is not discrete, the dominating measure $\nu$ if it exists,  may not be $\sigma$-finite. Indeed, since the $\phi_t$'s are 
ergodic for $\mu_u$, it follows that either $\mu_{u'} = \mu_{u''}$ or $\mu_{u'}$ and $\mu_{u''}$ are 
singular, for $\mu$-almost all $u', u''\in E$.  Thus, if ${\cal F}_\phi$ is ``too rich", this singularity feature implies that the measure $\nu$ may not be chosen to  be 
$\sigma$-finite.
\end{remark}


\section{Decomposability of Max-stable Processes}\label{sec:maxstable}

Max-stable processes are central objects in extreme value theory.  They arise in the limit
of independent maxima and thus provide canonical models for the dependence of the extremes (see e.g.~\cite{dehaan06extreme} and the references therein). Without loss of generality we focus here on $\alpha$-Fr\'echet processes.

Recall that a random variable $Z$ has an $\alpha$-Fr\'echet distribution, if
$\P(Z\le x) = \exp( -\sigma^\alpha x^{-\alpha})$ for all $x>0$ with some constant $\sigma>0$.
A process $Y=\{Y_t\}_{t\in T}$ is said to be $\alpha$-Fr\'echet if for all $n\in\mathbb N$, $a_i\ge 0,\  t_i\in T, i =1,\cdots,n,$ the
max-linear combinations $\max\{ a_i Y_{t_i},\ i=1,\cdots,n\} \equiv \bigvee_{i=1}^n a_iY_{t_i}$ are
$\alpha$-Fr\'echet. It is well known that a max-stable process is $\alpha$-Fr\'echet, if and only if it has $\alpha$-Fr\'echet marginals (de Haan~\cite{dehaan78characterization}). In the seminal paper~\cite{dehaan84spectral}, de Haan developed convenient
spectral representations of these processes.  An extremal integral representation, which parallels the
integral representations of S$\alpha$S processes, was developed by Stoev and Taqqu~\cite{stoev06extremal}.  

Let $Y=\indt Y$ be an $\alpha$-Fr\'echet ($\alpha>0$) process.  As in the S$\alpha$S case, if $Y$ is separable
in probability, it has the extremal representation
\begin{equation}\label{rep:extremal}
\{Y_t\}_{t\in T} \eqd {\Big\{} \Eint_S f_t(s) M^\vee_\alpha(\d s) {\Big\}}_{t\in T},
\end{equation}
where $\{f_t\}_{t\in T} \subset L_+^\alpha(S,\calB_S,\mu) = \{f\in L^\alpha(S,\calB_S,\mu): f\geq 0\}$ are non-negative deterministic functions, and where
$M_\alpha^\vee$ is an {\it $\alpha$-Fr\'echet random sup-measure with control measure $\mu$} (see~\cite{stoev06extremal} for more details).  The finite-dimensional distributions of $Y$ are characterized in terms of the spectral functions $f_t$'s as follows:
\equh\label{eq:cdf}
\proba(Y_{t_i} \le y_i,\ i=1,\cdots,n) = \exp {\Big\{} - \int_{S} {\Big(} \max_{1\le i\le n} \frac{f_{t_i}(s) }{y_i}
 {\Big)}^{\alpha} \mu(\d s) {\Big\}},
\eque
for all $y_i>0, t_i \in T,\ i=1,\cdots,n$.

The above representations of max-stable processes mimic those of S$\alpha$S processes~\eqref{rep:integral} and~\eqref{eq:fdd}.
The cumulative distribution functions and max-linear combinations of spectral functions, in the max-stable setting, play the role of characteristic functions and linear combinations in the sum-stable setting, respectively. In fact, the deep connection between the two classes of processes has been
clarified via the notion of {\em association} by Kabluchko~\cite{kabluchko09spectral} and Wang and Stoev~\cite{wang10association}, independently through different perspectives.

In the sequel, assume $0<\alpha<2$. An S$\alpha$S process $X$ and an $\alpha$-Fr\'echet process $Y$
 are said to be {\it associated} if they have a common spectral representation.  That is, if for {\em some} 
 non-negative $\{f_t\}_{t\in T} \subset L_+^\alpha(S,\calB_S,\mu)$, Relations 
 \eqref{rep:integral} and \eqref{rep:extremal} hold. The association is well defined in the following sense: 
 any other set of functions $\indt g\subset \lap(S,\calB_S,\mu)$ is a spectral representation of $X$, if and only if, it is 
 a spectral representation of $Y$ (see \cite{wang10association}, Theorem 4.1). 

\begin{remark} It is well known that $\wt Y = \indt{Y^\alpha}$ is a 1-Fr\'echet process (see e.g.~\cite{stoev06extremal}, Proposition 2.9). Moreover,
if~\eqref{rep:extremal} holds, then $\wt Y$ has spectral functions $\indt{f^\alpha}\subset L_+^1(S,\calB_S,\mu)$. Thus, the exponent
$\alpha>0$ plays no essential role in the dependence structure of $\alpha$-Fr\'echet processes.
Consequently, the notion of association (defined for $\alpha\in(0,2)$) can be used to study $\alpha$-Fr\'echet processes with
arbitrary positive $\alpha$'s.
\end{remark}

The association method can be readily applied to transfer decomposability results for S$\alpha$S processes to
the max-stable setting, where now sums are replaced by maxima.  Namely, let $Y=\{Y_t\}_{t\in T}$ be an
$\alpha$-Fr\'echet process.  If
\begin{equation}\label{e:max-decomp}
\{Y_t\}_{t\in T} \eqd {\Big\{} Y_t^{(1)} \vee \cdots \vee Y_t^{(n)}{\Big\}}_{t\in T},
\end{equation}
for some independent $\alpha$-Fr\'echet processes $Y^{(k)} = \{Y_t^{(k)}\}_{t\in T},\ i=1,\cdots,n$, then we say that
the $Y^{(k)}$'s are {\em components} of $Y$.  By the max-stability of $Y$,
\eqref{e:max-decomp} trivially holds if the $Y^{(k)}$'s are independent copies of
$\{ n^{-1/\alpha} Y_t\}_{t\in T}$. The constant multiples of $Y$ are referred to as trivial components of $Y$ and
 as in the S$\alpha$S case, we are interested in the structure of the non-trivial ones.

To illustrate the association method, we prove the max-stable counterpart of our main result Theorem~\ref{thm:1}. From the proof, we can see that the other results in the sum-stable setting have their natural max-stable 
counterparts by association. 
We briefly state some of these results at the end of this section.
\begin{theorem}\label{thm:2}
Suppose $\indt Y$ is an $\alpha$-Fr\'echet process with spectral representation \eqref{rep:extremal}, where
$F\equiv \indt f\subset L^\alpha_+(S,\calB_S,\mu)$.
Let $\{Y_t^{(k)}\}_{t\in T},\ k=1,\cdots,n$, be independent $\alpha$-Fr\'echet processes.  Then the
decomposition~\eqref{e:max-decomp} holds, if and only if there exist measurable functions
$r_k:S\to[0,1]$, $k=1,\cdots,n$, such that
\equh\label{eq:Yk}
 \indt{Y\topp k}\eqd\bccbb{\Eint_Sr_k(s)f_t(s)M_\alpha^\vee(\d s)}_{t\in T},\ \ k=1,\cdots,n.
\eque
In this case, $\summ k1n r_k(s)^\alpha = 1$, $\mu$-almost everywhere on $S$ and the $r_k$'s in \eqref{eq:Yk} can be chosen
to be $\rho(F)$-measurable, uniquely modulo $\mu$.
\end{theorem}
\begin{proof}
The `if' part follows from straight-forward calculation of the cumulative distribution functions~\eqref{eq:cdf}. To show the `only if' part, suppose~\eqref{e:max-decomp} holds and $Y\topp k$ has spectral functions $\indt{g\topp k}\subset\lap(V_k,\calB_{B_k},\nu_k)$, $k = 1,\dots,n$. Without loss of generality, assume $\{V_k\}_{k=1,\dots,n}$ to be mutually disjoint and define $g_t(v) = \summ k1n g_t\topp k(v)\ind_{V_k}\in\lap(V,\calB_V,\nu)$ for appropriately defined $(V,\calB_V,\nu)$ (see the proof of Theorem~\ref{thm:1}).

Now, consider the S$\alpha$S process $X$ associated to $Y$. It has spectral functions $\indt f$ and $\indt g$. Consider the S$\alpha$S processes $X\topp k$ associated to $Y\topp k$ via spectral functions $\indt {g\topp k}$ for $k = 1,\dots, n$. By checking the characteristic functions, one can show that $\{X\topp k\}_{k=1,\dots,n}$ form a decomposition of $X$ as in~\eqref{eq:decomposition}. Then, by Theorem~\ref{thm:1}, each S$\alpha$S component ${X\topp k}$ has a spectral representation~\eqref{eq:Xk} with spectral functions $\indt{r_kf}$. But we introduced $X\topp k$ as the S$\alpha$S process associated to $Y\topp k$ via spectral representation $\indt{g\topp k}$. Hence, $X\topp k$ has spectral functions $\indt{g\topp k}$ and $\indt{r_kf}$, and so does $Y\topp k$ by the association (\cite{wang10association}, Theorem 4.1). Therefore,~\eqref{eq:Yk} holds and the rest of the desired results follow.
\end{proof}
Further parallel results can be established by the association method. Consider a stationary $\alpha$-Fr\'echet process $Y$.
If $Y\topp k, k=1,\dots,n$ are independent stationary $\alpha$-Fr\'echet processes such that~\eqref{e:max-decomp} holds, then we
say each $Y\topp k$ is a {\it stationary $\alpha$-Fr\'echet component} of $Y$. The process $Y$ is said to be {\it indecomposable}, if it has no non-trivial stationary component. The following results on
{\em (mixed) moving maxima} (see e.g.~\cite{stoev06extremal} and~\cite{kabluchko09spectral} for more details) follow from
Theorem~\ref{thm:2} and the association method, in parallel to Corollary~\ref{coro:MMA1} on
(mixed) moving averages in the sum-stable setting.
\begin{corollary}\label{coro:max} 
The mixed moving maxima process
\[
{\{Y_t\}_{t\in\mathbb R^d}\eqd\bccbb{\Eint_{\mathbb R^d\times V}f(t+s,v)M_\alpha^\vee(\d s,\d v)}_{t\in\mathbb R^d}}
\]
is indecomposable, if and only if it has a moving maxima representation
\[
{\{Y_t\}_{t\in\mathbb R^d}\eqd \bccbb{\Eint_{\mathbb R^d}f(t+s)M_\alpha^\vee(\d s)}_{t\in\mathbb R^d}}\,.
\]
\end{corollary}


\section{Proof of Theorem~\ref{thm:1}}\label{sec:proofs}
We will first show that Theorem~\ref{thm:1} is true when $\indt f$ is minimal (Proposition~\ref{prop:minimal}), and then we complete the proof by relating a  general spectral representations to a minimal one. This technique is standard in the literature of representations of S$\alpha$S processes (see e.g.~Rosi\'nski~\cite{rosinski95structure}, Remark 2.3).
We start with a useful lemma.
\begin{lemma}\label{lem:unique}
Let $\indt f\subset L^\alpha(S,\calB_S,\mu)$ be a minimal representation of an S$\alpha$S process.
For any two bounded $\calB_S$-measurable functions $r\topp1$ and $r\topp2$, we have
\[
\bccbb{\int_Sr\topp1f_t\d M_\alpha}_{t\in T} \eqd\bccbb{\int_Sr\topp2f_t\d M_\alpha}_{t\in T}\,,
\]
if and only if $|r\topp1| = |r\topp2|\ \mbox{ modulo }\mu$.
\end{lemma}
\begin{proof} The 'if' part is trivial. We shall prove now the 'only if' part.
Let $S\topp k:={\rm supp}(r\topp k),\ k=1,2$ and note that since $\indt f$ is minimal,
then $\indt {r\topp k f}$, are minimal representations, restricted to $S\topp k$, $k=1,2$, respectively.
Since the latter two representations correspond to the same process, by Theorem 2.2 in \cite{rosinski95structure},
there exist a bi-measurable, one-to-one and onto point mapping $\Psi :S\topp 1\to S\topp 2$ and
a function $h:S\topp1\to\mathbb R\setminus\{0\}$, such that, for all $t\in T$,
 \begin{equation}\label{e:f_12}
 r\topp1(s) f_t(s) = r\topp2\circ\Psi(s) f_t\circ \Psi(s)h(s)\,,\mbox{ almost all } s\in S\topp1,
 \end{equation}
 and
 \equh\label{eq:h}
 \ddfrac\mu\Psi\mu = |h|^\alpha\,, \mu\mbox{-almost everywhere.}
 \eque
 It then follows that, for almost all $s\in S\topp 1$,
 \begin{equation}\label{e:ratio}
 \frac{f_{t_1}(s)}{f_{t_2}(s)} =  \frac{r\topp 1(s) f_{t_1}(s)}{r\topp1(s) f_{t_2}(s)} = \frac{f_{t_1}\circ\Psi(s)}{f_{t_2}\circ\Psi(s)}.
 \end{equation}
 Define $ R_\lambda(t_1,t_2) = \{s\, :\,  f_{t_1}(s)/f_{t_2}(s)\le \lambda\}$
 and note that by \eqref{e:ratio},
 for all $A\equiv R_\lambda(t_1,t_2)$,
 \equh\label{eq:Delta}
 \mu\spp{\Psi(A\cap S\topp 1)\Delta(A\cap S\topp2)} = 0\,.
 \eque
In fact, one can show that Relation \eqref{eq:Delta} is also valid for all $A \in \rho(F)\equiv\sigma(R_\lambda(t_1,t_2):\lambda\in\mathbb R, t_1,t_2\in T)$.
Then, by minimality, \eqref{eq:Delta} holds for all $A\in\calB_S$. In particular, taking $A$ equal to $S\topp1$ and
$S\topp2$, respectively, it follows that $\mu(S\topp1\Delta S\topp2) = 0$. Therefore, writing $\wt S \defe S\topp1\cap S\topp2$, we have
 \equh\label{eq:wtS}
 \mu(\Psi(A\cap \widetilde S) \Delta (A \cap \widetilde S)) = 0,\ \ \mbox{ for all }A\in \calB_S\,.
 \eque
 
This implies that $\Psi(s) = s$, for $\mu$-almost all $s\in \wt S$. To see this, let $\calB_{\wt S} = \calB_S\cap \wt S$ denote the $\sigma$-algebra $\calB_S$ restricted to $\wt S$. Observe that for all
$A\in\calB_{\wt S}$, we have $\ind_A = \ind_A\circ\Psi$, for $\mu$-almost all $s\in\wt S$, and trivially $\sigma(\ind_A:A\in\calB_{\wt S}) = \calB_{\wt S}$.  Thus, by the second
part of Proposition 5.1 in~\cite{rosinski06minimal}, it follows that $\Psi(s) = s$ modulo $\mu$ on $\wt S$. This and \eqref{eq:h} imply
that $h(s) \in \{\pm1\}$, almost everywhere. Plugging $\Psi$ and $h$ into~\eqref{e:f_12} yields the desired result.
\end{proof}

\begin{proposition}\label{prop:minimal} Theorem~\ref{thm:1} is true when $\indt f$ is minimal.
\end{proposition}
\begin{proof} {\em We first prove the 'if' part}. The result follows readily 
by using characteristic functions.  Indeed, suppose that the $X^{(k)} = \{X_t^{(k)}\}_{t\in T},$
 $k=1,\dots,n$ are independent and have representations as in \eqref{eq:Xk}.  Then, for all 
 $a_j\in\bbR,\ t_j\in T,\ j=1,\cdots, m,$ we have
\begin{multline}\label{e:lemma-1}
\esp\exp\bpp{i\summ j1ma_jX_{t_j}} = \exp\bpp{-\int_S\babs{\summ j1ma_jf_{t_j}}^\alpha\d\mu}\\
= \prod_{k=1}^n\exp\bpp{-\int_S\babs{\summ j1ma_jr_kf_{t_j}}^\alpha\d\mu} = \prod_{k=1}^n\esp\exp\bpp{i\summ j1ma_jX_{t_j}\topp k}\,,
\end{multline}
where the second equality follows from the fact that $\summ k1n|r_k(s)|^\alpha = 1,$ for  $\mu$-almost all 
$s\in S.$  Relation \eqref{e:lemma-1} implies the decomposition \eqref{eq:decomposition}.

{\em We now prove the 'only if' part}.  Suppose that \eqref{eq:decomposition} holds and
let $\indt{f\topp k}\subset L^\alpha(V_k,\calB_{V_k},\nu_k),\ k=1,\dots,n$ be representations for the independent components
$\indt {X\topp k}, k=1,\dots,n$, respectively, and without loss of generality,  assume that $\{V_k\}_{k=1,\dots,n}$ are mutually disjoint.
Introduce the measure space $(V,\calB_V,\nu)$, where $V:= \bigcup_{k=1}^nV_k$,
$\calB_V:= \{\bigcup_{k=1}^n A_k,\ A_k\in \calB_{V_k},\ k=1,\dots,n\}$ and $\nu(A):= \summ k1n\nu_k(A\cap V_k)$ for
all $A\in \calB_V$.

By \eqref{eq:decomposition}, it follows that
 $\{X_t\}_{t\in T} \eqd\{\int_V g_t \d \wb M_\alpha\}_{t\in T}$, with $g_t(u)\defe\summ k1n f_t\topp k(u)\ind_{V_k}(u)$ and $\wb M_\alpha$ an S$\alpha$S random
 measure on $(V,\calB_V)$ with control measure $\nu$.

 Thus, $\{f_t\}_{t\in T} \subset L^\alpha(S,\calB_S,\mu)$ and $\{g_t\}_{t\in T} \subset L^\alpha(V,\calB_V,\nu)$ are two representations of the
 same process $X$, and by assumption the former is {\em minimal}. Therefore, by Remark 2.5 in~\cite{rosinski95structure},
 there exist modulo $\nu$ unique functions $\Phi:V\to S$ and $h:V\to\mathbb R\setminus{\{0\}}$, such that, for all $t\in T$,
\equh\label{eq:ftk}
g_t(u) = h(u)f_t\circ\Phi(u)\,,\mbox{ almost all }u\in V\,,
\eque
where moreover $\mu = \nu_h\circ\Phi\inv$ with $\d\nu_h = |h|^\alpha\d\nu$.

Recall that $V$ is the union of mutually disjoint sets $\{V_k\}_{k=1,\dots,n}$. For each $k=1,\dots,n$, let $\Phi_k:V_k \to S_k := \Phi(V_k)$ be the restriction
of $\Phi$ to $V_k$, and define the measure $\mu_k(\cdot) \defe\nu_{h,k}\circ\Phi_k\inv(\,\cdot\,\cap S_k)$ on $(S,\calB_S)$ with
$\d\nu_{h,k}\defe|h|^\alpha\d\nu_k$. Note that $\mu_k$ has support $S_k$, and the Radon--Nikodym derivative $\d\mu_k/\d\mu$ exists.
We claim that~\eqref{eq:Xk} holds with $r_k\defe(\d\mu_k/\d\mu)^{1/\alpha}$. To see this, observe that
for all $ m\in\mathbb N, a_1, \ldots, a_m \in\mathbb R, t_1, \ldots, t_m \in T$, 
\[
\int_S\babs{\summ j1m a_jr_kf_{t_j}}^\alpha\d\mu = \int_{S_k}\babs{\summ j1m a_jf_{t_j}}^\alpha\d\mu_k = \int_{V_k}\babs{\summ j1ma_jhf_{t_j}\circ\Phi_k}^\alpha\d\nu_k\,,
\]
which, combined with~\eqref{eq:ftk}, yields \eqref{eq:Xk} because $g_t\vert_{V_k} = f_t^{(k)}$.

Note also that $\summ k1n\mu_k = \mu$ and thus $\summ k1nr_k^\alpha=1$.
This completes the proof of  part {\it (i)} of Theorem \ref{thm:1} in the case when $\{f_t\}_{t\in T}$ is minimal.

To prove part {\it (ii)}, note that the $r_k$'s above are in fact non-negative and $\calB_S$-measurable.  Note also that by minimality,
the $r_k$'s have versions $\wt r_k$'s that are $\rho(F)$-measurable,  i.e.\ $r_k = \wt r_k$ modulo $\mu$.
Their uniqueness follows from Lemma \ref{lem:unique}.
\end{proof}
\begin{proof}[{\it Proof of Theorem \ref{thm:1}}]
\noindent {\it (i)} The `if' part follows by using characteristic functions as in the proof of Proposition~\ref{prop:minimal} above.

\noindent{Now, we prove the `only if' part.} Let $\{\wt f_t\}_{t\in T} \subset L^\alpha (\wt S,\calB_{\wt S}, \wt \mu)$ be a minimal representation of $X$.  As in the proof of Proposition~\ref{prop:minimal}, by Remark 2.5 in~\cite{rosinski95structure},
 there exist modulo $\mu$ unique functions $\Phi:S\to \wt S$ and $h:S\to\mathbb R\setminus{\{0\}}$, such that,
 for all $t\in T$,
\equh\label{eq:ftk-thm-1}
 f_t(s) = h(s)\wt f_t\circ\Phi(s)\,,\mbox{ almost all }s\in S,
\eque
where $\wt\mu = \mu_h\circ\Phi\inv$ with $\d \mu_h = |h|^\alpha\d\mu$.

Now, by Proposition~\ref{prop:minimal}, if the decomposition \eqref{eq:decomposition} holds, then there exist unique non-negative functions $\wt r_k,\ k=1,\cdots,n$, such that
\begin{equation}\label{e:thm-1.1}
\{ X^{(k)}_t\}_{t\in T} \stackrel{d}{=} {\Big\{} \int_{\wt S} \wt r_k \wt f_t \d \wt M_\alpha  {\Big\}}_{t\in T},\ \ k=1,\cdots,n,
\end{equation}
and $\sum_{k=1}^n \wt r_k^\alpha = 1$ modulo $\wt \mu$.  Here $\wt M_\alpha$ is an  S$\alpha$S measure on
$(\wt S,\calB_{\wt S})$ with control measure $\wt \mu$. Let $r_k(s) := \wt r_k \circ \Phi(s)$ and note that
by using \eqref{eq:ftk-thm-1} and a change of variables,  for all $a_j \in \bbR, t_j\in T,\ j=1,\cdots,m$, we obtain
\begin{equation}\label{e:thm-1.2}
\int_{S} {\Big|}\sum_{j=1}^m a_j r_k(s) f_{t_j}(s) {\Big|} \mu(\d s) = \int_{\wt S} {\Big|}\sum_{j=1}^m a_j
\wt r_k(s) \wt f_{t_j}(s) {\Big|} \wt \mu(\d s).
\end{equation}
This, in view of Relation \eqref{e:thm-1.1}, implies \eqref{eq:Xk}. Further, the fact that
$\sum_{k=1}^n\wt r_k^\alpha =1$ implies $\sum_{k=1}^nr_k^\alpha =1$, modulo $\mu$, because
the mapping $\Phi$ is non-singular, i.e.\ $\wt \mu\circ \Phi^{-1} \sim \mu$.
 This completes the proof of part {\it (i)}.

\medskip We now focus on proving part {\it (ii)}.  Suppose that \eqref{eq:Xk} holds for two choices of $r_k$, namely
$r_k'$ and $r_k''$.  Let also $r_k'$ and $r_k''$ be non-negative and measurable with respect to $\rho(F)$.  We claim that 
\equh\label{eq:rhoF}
\rho(F)\sim \Phi\inv(\rho(\wt F))
\eque
and defer the proof to the end. Then, since the minimality implies that $\calB_{\wt S}\sim\rho(\wt F)$. 
$r_k'$ and $r_k'' $ are measurable with respect
to $\rho(F) \sim \Phi^{-1}(\calB_{\wt S})$. Now, Doob--Dynkin's lemma (see e.g.~Rao~\cite{rao05conditional}, p.~30) implies that
\begin{equation}\label{e:thm-1.3}
 r_k'(s) = \wt r_k'\circ \Phi(s)\ \ \mbox{ and } \ \  r_k''(s) = \wt r_k'' \circ \Phi(s),\ \ \mbox{ for $\mu$ almost all $s$},
\end{equation}
where $\wt r_k'$ and $\wt r_k''$ are two $\calB_{\wt S}$-measurable functions.
By using the last relation and a change of variables, we obtain that \eqref{e:thm-1.2} holds with $(r_k,\wt r_k)$ replaced by $(r_k',\wt r_k')$ and $(r_k'',\wt r_k'')$, respectively.  Thus both $\indt{\wt r_k'\wt f}$ and
$\indt{\wt r_k''\wt f}$ are representations of the $k$-th component of $X$. Since $\{\wt f_t\}_{t\in T}$ is a minimal representation of $X$,
Lemma~\ref{lem:unique} implies that $\wt r_k' = \wt r_k''$ modulo $\wt \mu$. This, by \eqref{e:thm-1.3} and the
non-singularity of $\Phi$ yields $r_k' = r_k''$ modulo $\mu$. 

It remains to prove~\eqref{eq:rhoF}
Relation \eqref{eq:ftk-thm-1} and the fact that $h(s)\not=0$ imply that
for all $\lambda$ and $t_1,t_2\in T$,
$
 \{ f_{t_1}/f_{t_2} \le \lambda \}  = \Phi^{-1} (\{ \wt f_{t_1}/\wt f_{t_2} \le \lambda \} )\mbox{ modulo }\mu.
$
Thus the classes of sets ${\cal C}:=  \sccbb{ \{f_{t_1}/f_{t_2} \le \lambda \},\ t_1,t_2\in T,\ \lambda\in \bbR}$
and $\wt {\cal C}:=\sccbb{ \Phi^{-1} (\{ \wt f_{t_1}/\wt f_{t_2} \le \lambda \}),\ t_1,t_2\in T,\ \lambda\in \bbR} $ are equivalent. That is, for all $A \in {\cal C}$, there exists $\wt A\in \wt{\cal C}$, with $\mu(A\Delta \wt A) = 0$ and vice versa.

Define
$$
 \wt {\cal G} = \bccbb{\Phi^{-1}(A): A\in \rho(\wt F)\mbox{ such that\,} \mu(\Phi^{-1}(A) \Delta B) = 0  \mbox{ for some } B\in \sigma({\cal C})}.
$$
Notice that $\wt {\cal G}$ is a $\sigma$-algebra and since $\wt {\cal C} \subset \wt {\cal G} \subset
\Phi^{-1}(\rho(\wt F))$, we obtain that $\sigma(\wt {\cal C}) = \Phi^{-1}(\rho(\wt F)) \equiv
\wt {\cal G}$.  This, in view of definition of $ \wt {\cal G}$, shows that for all $\wt A \in \sigma(\wt {\cal C})$, exists
$A\in \sigma({\cal C})$ with $\mu(A\Delta \wt A) = 0$.  In a similar way one can show that each element of
$\sigma({\cal C})$ is equivalent to an element in $\sigma(\wt {\cal C})$, which completes the proof of the
desired equivalence of the $\sigma$-algebras.
\end{proof}

\noindent{\bf Acknowledgment}
 Yizao Wang and Stilian Stoev's research were partially supported by the NSF grant DMS--0806094 at the University of Michigan, Ann Arbor. The authors were grateful to Zakhar Kabluchko for pointing out two mistakes in a previous version and for many helpful discussions. They also thank two anonymous referees for helpful comments and suggestions.


\def\cprime{$'$} \def\polhk#1{\setbox0=\hbox{#1}{\ooalign{\hidewidth
  \lower1.5ex\hbox{`}\hidewidth\crcr\unhbox0}}}

\end{document}